\documentclass[12pt]{article}
\usepackage{amsfonts}
\usepackage{amsmath}
\usepackage{amsthm}
\usepackage{graphicx}
\usepackage{hyperref}
\usepackage[latin1]{inputenc}
\usepackage[
backend=biber,
style=alphabetic,
sorting=ynt
]{biblatex}

\title{The Mondrian Puzzle: A Connection to Number Theory}
\author{Cooper O'Kuhn}
\date{\today}

\begin{document}
\maketitle

\begin{abstract}
We obtain partial progress towards answering the question of whether the quantity defined in the Mondrian Puzzle can ever equal 0. More specifically, we obtain a nontrivial lower bound for the cardinality of the set $\{n\leq x: M(n)\neq0 \}$ where $M(n)$ is the quantity appearing in the Mondrian Puzzle and $x$ is the usual quantity that one thinks of as tending to infinity. More surprisingly, we do so by use of number theoretic techniques in juxtaposition to the innately geometric nature of the problem.
\end{abstract}

\section{Introduction}
The Mondrian Puzzle is a problem based on the artwork of the Dutch artist Piet Mondrian. His paintings are quite unique, simply consisting of primary-colored configurations of tesselated rectangles. The idea of the puzzle is that an art critic has ordered Mondrian to only create paintings whose rectangles are all incongruent to one another and only have integer side lengths. Furthermore, he can only use a square canvas whose side length is also an integer. Aggravated, Mondrian still wants to create works whose rectangles are all as close as possible in area (note: the art critic said nothing about the area of the rectangles). This leads one to make the following definitions.

Let $n\geq3$ be a natural number. Let $M(n)$ be the minimal possible difference in areas of the largest and smallest rectangles in a set of incongruent, integer-sided rectangles that tesselate an $n$ by $n$ square. It is an open question whether there exists an integer $n$ such that $M(n)=0$, but it is widely believed that such an $n$ should not exist. In this paper, we say something about the density of numbers $n$ that satisfy $M(n)\neq0$ in a given range $[1,x].$ Of course conjecturally, the density of these numbers should be 1. Instead, we get a lower bound of this quantity that roughly takes the shape $\frac{c\log(\log(x))}{\log(x)}$ for some constant $c>0$. More specifically, we have the following. 

\newtheorem*{theorem*}{Theorem}
\newtheorem{lemma}{Lemma}

\begin{theorem*} For all $x\geq3$ and all $\epsilon>0$, we have
\begin{equation}
|\{n\leq x: M(n)\neq 0 \}|\geq \frac{C_\epsilon x\log(\log(x))}{\log(x)}\left(1+O_\epsilon\left(\frac{\log(\log(x))}{\log(x)}\right)\right),
\end{equation}
 where $$C_\epsilon =\frac{1}{e^\gamma(2\log(2)+\epsilon)},$$ and $\gamma$ is the Euler-Mascheroni constant.
\end{theorem*}
If one divides both sides of (1) by $x$, one can view this as a statement regarding density. This problem intuitively seems very geometric in nature. However, as will become clear shortly, this problem has strong ties to very deep results in Number Theory, and thus the proof of (1) is very number theoretic in nature. 

\section{Notation} We adopt the following conventions throughout the paper. We let $x$ be some parameter tending to infinity. We further let $\epsilon$ be some small real quantity that may not be the same in each occurrence. We let $\alpha,a,b,d,h,j,k,l,n$ and $r$ exclusively represent positive integers, and $z$ represents a real quantity. Furthermore, $|\{.\}|$ denotes the cardinality of a set. We say $d|n$ if $d$ divides $n$, and we say $d||n$ if $d$ exactly divides $n$ in the sense that $d^2$ does not divide $n$; $\tau_2(n)$ denotes the number of divisors of $n$ (ex: $\tau_2(6)=4$ since $1|6, 2|6, 3|6,$ and $6|6$), and $\tau_2^*(n)$ denotes the number of unique representations of $n$ as a product of two natural numbers (ex: $\tau_2^*(6)=2$ since $6=1\times6=2\times3$). Notice the relationship $\tau_2^*(n)=\frac{\tau_2(n)+s(n)}{2},$ where $s(n)$ is the characteristic function of square numbers, due to the fact that the factors of $n$ are symmetric about $\sqrt{n}$. We say that a number a natural number $n$ is $z$-rough if every prime that divides $n$ is greater than $z$.

We further define $O(1)$ to be any quantity that remains bounded as $x$ tends to infinity, and we define $o(1)$ to be any quantity that tends to zero as $x$ tends to infinity. We say that $O(1)U=O(U)$ and $o(1)U=o(U)$ for any quantity $U$. We further write $O_z$ if the implied constant depends on $z$.  

\section{Proof of the Main Theorem}

The proof goes as follows. We are first and most importantly tasked with deriving a subset of the set $\{n\leq x: M(n)\neq0\}$ whose elements have a specific number theoretic structure. After this, we use a multitude of (generally simple) number theoretic manipulations to further reduce our problem to one of estimating the cardinality of $F(x,z)$, the set of all $z$-rough integers less than or equal to $x$, or a set of the form 
\begin{equation}
F(x,z):=\{n\leq x: (\forall d>1)(d|n\Rightarrow d>z)\}.
\end{equation}
The main challenge in reducing to a set like $F(x,z)$ is reducing the conditions of the set in such a way that the $z$ parameter is 1) completely independent of $n$ and 2) a nice smooth function, and thus most of the manipulations carried out below are done in an effort to do this. 

We first derive our subset. 
\begin{lemma} We have 
\begin{equation}
\{n\leq x: (\forall d<n^2)(d|n^2 \Rightarrow d\tau_2(d)<n^2)\}\subseteq\{n\leq x: M(n)\neq0\}.
\end{equation}
\end{lemma} 
\begin{proof}
Arguing indirectly for the moment, suppose that there exists a positive integer $n\geq 3$ such that $M(n)=0$. Thus, there exists a set of incongruent integer-sided rectangles all of the same area, say, $d$ that tesselate an $n$ by $n$ square, a set which we denote by $S$. Notice that we may assume without loss of generality that $d$ is strictly less than $n^2$ since this corresponds to the set of rectangles that just contains the $n$ by $n$ square. Now, let $A$ be the collective area of all the rectangles in $S$. Since the rectangles in $S$ only have area $d$, one can see that we have
\begin{equation}
A=d|S|.
\end{equation}
Certainly, if $M(n)=0$, we must also have 
$$A=n^2$$
since this is the area of an $n$ by $n$ square. As a consequence of this requirement and (4), we have that 
$$d|n^2.$$ 

In an effort to aid a later argument, we can calculate an upper bound for $A$. In view of (4), we are left to bound $|S|$. To do this, map every rectangle in $S$ to the ordered pair containing its corresponding base and height, $(b_i,h_i)$. Since every rectangle in $S$ is incongruent to each other, the mapping from $S$ to the elements of the set $$\{(b_i,h_i):b_ih_i=d:b_i\geq h_i\}$$ is completely injective. Thus, we have 
\begin{equation}
|S|\leq |\{(b_i,h_i):b_ih_i=d:b_i\geq h_i\}|.
\end{equation}
One can see that the right hand side of (5) is the set of all ordered pairs of natural numbers whose product is $d$, and thus it has cardinality $\tau_2^*(d)$. We therefore have
\begin{equation}
A\leq d\tau_2^*(d).
\end{equation} 
Now, for the sake of contradiction, suppose that we also knew that the condition 
\begin{equation}
\forall d<n^2, d|n^2 \Rightarrow d\tau_2^*(d)<n^2
\end{equation}
held as in the lemma. This would imply that $A<n^2$ for any choice of $d$ by (6). This is a contradiction since we have shown it necessary that $A=n^2$ for some $d|n^2$ in order for $M(n)=0$ to hold, and thus $M(n)$ cannot equal 0 for any value of $n$ satisfying (7). Therefore, we have 
$$\{n\leq x: (\forall d<n^2)(d|n^2 \Rightarrow d\tau_2^*(d)<n^2))\}\subseteq\{n\leq x: M(n)\neq0\}$$
which concludes the proof. 

\end{proof}
 We will now obtain a lower bound for $$|\{n\leq x : (\forall d<n^2)(d|n^2 \Rightarrow d\tau^*_2(d)<n^2) \}|,$$ which will yield a lower bound for the set in (1) by Lemma 1 and the fact that for all sets $A$ and $B$, if $A\subseteq B$, then $|A|\leq|B|$. We will be repeatedly and implicitly using the fact that if one places ``stronger" restraints on a set (in our case, we are mainly referring to inputing stronger upper or lower bounds into the conditions of a set), one obtains a subset of the original set. Before this, we first make the change of variables $d=\frac{n^2}{k}$. Since this is equivalent to the statement $k=\frac{n^2}{d}$ which is an integer by the hypothesis $d|n^2$, we have that $k|n^2$. So, 
 \begin{multline*}
 |\{n\leq x: (\forall d<n^2)(d|n^2 \Rightarrow d\tau_2^*(d)<n^2))\}| = \\
 |\{n\leq x: (\forall k>1)(k|n^2 \Rightarrow \frac{n^2}{k}\tau_2^*\left(\frac{n^2}{k}\right)<n^2)\}|, 
  \end{multline*}
or equivalently
\begin{multline}   
|\{n\leq x: (\forall d<n^2)(d|n^2 \Rightarrow d\tau_2^*(d)<n^2))\}|= \\ 
|\{n\leq x: (\forall k>1)(k|n^2 \Rightarrow \tau_2^*\left(\frac{n^2}{k}\right)<k)\}|.
\end{multline}  
Recall the relation $\tau^*_2(n)=\frac{\tau_2(n)+s(n)}{2}$. Since $\tau_2(n)\geq 2$, for all $n$, this relation implies the inequality $\tau^*_2(n)\leq\tau_2(n)$ for all $n$. Thus, we can make the following reduction: 
\begin{multline}
|\{n\leq x: (\forall k>1)(k|n^2 \Rightarrow \tau_2^*\left(\frac{n^2}{k}\right)<k)\}|\geq \\
|\{n\leq x: (\forall k>1)(k|n^2 \Rightarrow \tau_2\left(\frac{n^2}{k}\right)<k)\}|.
\end{multline}
This is largely advantageous for us in our efforts to reduce as will soon become clear. Informally, it is because $\tau_2$ is a much more natural function to deal with than $\tau_2^*$. One begins to notice the extent of this phenomenon upon proving the following lemma in that the analogous statement for $\tau_2^*$ is significantly more tedious.

\begin{lemma} For all $d,n\in\mathbb{N}$, we have $d|n \Rightarrow\tau_2(d)\leq\tau_2(n)$.
\end{lemma}

\begin{proof} Let $L(k)$ be the set of all divisors of $k$ so that 
$$|L(k)|=\tau_2(k).$$ 
Since any divisor of $d$ must also divide $n$ by the relation $d|n$, we have that 
$$L(d)\subseteq L(n).$$
Thus, we further have 
$$|L(d)|\leq|L(n)|,$$ 
or equivalently 
$$\tau_2(d)\leq\tau_2(n).$$
\end{proof}
This lemma allows us to simplify the conditions of the set further: 
\begin{multline} 
|\{n\leq x: (\forall k>1<n^2)(k|n^2 \Rightarrow \tau_2\left(\frac{n^2}{k}\right)<k)\}|\geq \\
|\{n\leq x: (\forall k>1)(k|n^2 \Rightarrow \tau_2(n^2)<k)\}|
\end{multline}

We use the following lemma to simplify the conditions of our set by replacing one occurrence of $n^2$ with $n$ in an effort to further conform the right hand side of (10) to a set like the one in (2):
\begin{lemma} If $z>1$ is any real number, then the statement ``$n$ is $z$-rough" is equivalent to the statement ``$n^2$ is $z$-rough."
\end{lemma}
\begin{proof} Let $D(n)$ denote the smallest non-unitary divisor of $n$. Note that $n$ is $D(n)$-rough for every $n$. We first show that $D(n)$ is prime by means of infinite descent.

Suppose that $D(n)$ was composite. Then there must exist a factorization 
$$D(n)=ab$$ 
where $a$ and $b$ are both positive integers and neither are equal to 1. However if this was the case, we would have $a<D(n)$ and since $a|n$, $D(n)$ would no longer be the smallest divisor of $n$. This is a contradiction, and thus  $D(n)$ must be prime. 

 To conclude the proof, we need to show that 
 $$D(n)\notin L(n^2)-L(n),$$
 or informally that $D(n)$ is not in the set of new factors created upon squaring $n$, since this would imply that 
  $$D(n)=D(n^2).$$
 This can be seen by noting that every element of the set $L(n^2)-L(n)$ is not square-free, and every prime by definition must be square-free.
\end{proof}
 Substituting, we have that 
 \begin{multline}
 |\{n\leq x: (\forall k>1)(k|n^2 \Rightarrow  \tau_2(n^2)<k)\}| = \\
 |\{n\leq x: (\forall k>1)(k|n \Rightarrow  \tau_2(n^2)<k)\}|.
\end{multline}
The following lemma allows us to make everything in terms of smooth functions, another necessary condition for the application of sieve methods. 

\begin{lemma}[Divisor Bound] For all $\epsilon>0$, there exists $n_0=n_0(\epsilon)$ which depends only on $\epsilon$ such that for all $n\geq n_0$, we have $$\tau_2(n)\leq n^{\frac{\log(2)+\epsilon}{\log(\log(n))}}.$$ 
\end{lemma}
\begin{proof} See [1, page 294] 
\end{proof} 

For convenience, we define $g_\epsilon(x):=x^\frac{\log(2)+\epsilon}{\log(\log(x))}$ so that Lemma 4 can be restated as $\forall n\geq n_0, \tau_2(n)\leq g_\epsilon(n)$. Using this, we have 
\begin{multline}
|\{n\leq x: (\forall k>1)(k|n \Rightarrow \tau_2(n^2)<k\}|\geq \\
|\{n_0\leq n\leq x: (\forall k>1)(k|n \Rightarrow g_\epsilon(n^2)<k)\}|
\end{multline}
for some $\epsilon>0$ to be chosen at one's disposal. (Note: we may assume that $x\geq n_0$ since we are assuming the implied constant in the theorem is sufficiently large). 
Lastly, we use the relation 
$$\forall n\leq x, g_\epsilon(n^2)\leq g_\epsilon(x^2)$$
 (since $x$ is sufficiently large) to make the conditions of the set on the right hand side of (12) further independent of $n$: 
\begin{multline*}
|\{n_0\leq n\leq x: (\forall k>1)(k|n \Rightarrow  g_\epsilon(n^2)<k)\}|\geq \\ 
|\{n_0\leq n\leq x: (\forall k>1)(k|n \Rightarrow  g_\epsilon(x^2)<k)\}|,
\end{multline*}
and, in view of the definition of $F(x,z)$ (and the fact that $n_0$ is a constant), we equivalently have 
\begin{equation}
|\{n_0\leq n\leq x: (\forall k>1)(k|n \Rightarrow  g_\epsilon(n^2)<k)\}|\geq \\ 
|F(x,g_\epsilon(x^2))| + O(1)
\end{equation}
 This finally allows us to use Sieve Theory to count our set. In order to do this, we appropriately need a lemma from Sieve Theory.
 \begin{lemma}[Fundamental Lemma of the Selberg Sieve]
 For all $x>2$ and for all $1<z\leq x$, we have 
 $$|F(x,z)|=x\prod_{p\leq z}\left(1-\frac{1}{p}\right)\left(1+O\left(e^\frac{-\log(x)}{2\log(z)}\right)\right)$$
\end{lemma}
\begin{proof} See [2, page 208-209].
\end{proof}

Using Lemma 5 with $z=g_\epsilon(x^2)$ yields 
\begin{equation}
\begin{split}
|F(x,g_\epsilon(x^2))| &= x\prod_{p\leq g_\epsilon(x^2)}\left(1-\frac{1}{p}\right)\left(1+O\left(e^\frac{-\log(x)}{2\log(g_\epsilon(x^2))}\right)\right) \\
											         &=x\prod_{p\leq x^{\frac{2\log(2)+\epsilon}{\log(2\log(x))}}}\left(1-\frac{1}{p}\right)\left(1+O\left(\frac{1}{\log(x)}\right)\right)
\end{split}
\end{equation}
Using Mertens' estimate 
$$\prod_{p<z}\left(1-\frac{1}{p}\right)=\frac{e^{-\gamma}}{\log(z)}\left(1+O\left(\frac{1}{\log(z)}\right)\right),$$ 
(see [2, page 20]), we have 
\begin{equation}
|F(x,g_\epsilon(x^2))| = \frac{\frac{1}{e^\gamma (2\log(2)+\epsilon)}x\log(\log(x))}{\log(x)}\left(1+O\left(\frac{\log(\log(x))}{\log(x)}\right)\right).
\end{equation}
Thus, since we have shown using (3) and (8)--(13) that 
\begin{equation*}
|F(x,g_\epsilon(x^2))|\leq |\{n\leq x: M(n)\neq0\}|,
\end{equation*} 
we have that
\begin{equation*}
|\{n\leq x: M(n)\neq0 \}|\geq \frac{C_\epsilon x\log(\log(x))}{\log(x)}\left(1+O\left(\frac{\log(\log(x))}{\log(x)}\right)\right).
\end{equation*}

\section{The Way Forward}
The lower bound for $|\{n\leq x: M(n)\neq0\}|$ we have proven is much smaller than what is expected to be the truth (since every natural number $n$ is believed to satisfy $M(n)\neq0$, this quantity should be $x$). The majority of the loss seems to be acquired from the manner in which the divisor bound was applied. Thus, it only seems natural to, in order to improve our result, bound the set on the right hand side of (11) more directly, since this is the set to which the divisor bound is directly applied in the above argument. The following is the method by which it is suggested one should go about this.

Let
$$
T_j(n) = \left\{
\begin{array}{ll}
1 & \quad \tau_2(n)=j \\
0 & \quad otherwise
\end{array}
\right.
,$$
and let 
$$
P_z(n) = \left\{
\begin{array}{ll}
1 & \quad (\forall d>1)(d|n \Rightarrow d>z) \\
0 & \quad otherwise
\end{array}
\right.
.$$

We can write the indicator function for the condition $$(\forall d>1)(d|n\Rightarrow d>\tau_2(n^2))$$ in terms of the functions $T_j$ and $P_j$ and obtain an expression for the right hand side of (11) by the following:
$$|\{n\leq x: (\forall d>1) (d|n\Rightarrow d>\tau_2(n^2))\}|=\sum_{n\leq x} \sum_{j\in I(x)} T_j(n^2)P_j(n),$$
where 
\begin{equation*}
I(x):=\{j:(\exists n)(n\leq x:\tau_2(n^2)=j)\}
\end{equation*}
is the set of all values the function $\tau_2(n^2)$ can take for $n\leq x$. Informally, one can see that the innermost sum is the rewritten expression for the indicator function of the desired condition. 

By the Divisor Bound, one has that 
\begin{equation*}
I(x)\subseteq\{j: 2\leq j\leq g_\epsilon(x^2)\}.
\end{equation*}
Since the sum 
\begin{equation*}
\sum_{j\in I(x)}T_j(n^2)P_j(n)
\end{equation*}
is an indicator function and thus can only take the values 0 or 1 by definition, we can  substitute in this superset of $I(x)$ while maintaining equality:  
\begin{equation*}
\sum_{n\leq x} \sum_{j\in I(x)} T_j(n^2)P_j(n) = \sum_{n\leq x} \sum_{2\leq j \leq g_\epsilon(x^2)} T_j(n^2)P_j(n) = \sum_{2\leq j \leq g_\epsilon(x^2)} \sum_{n\leq x} T_j(n^2)P_j(n).
\end{equation*}

We further reduce to only counting those numbers that are square-free, a reduction which is beneficial for a number of reasons. Most importantly, we have that 
$$\sum_{n\leq x}\mu^2(n)=\left(\frac{6}{\pi^2}+o(1)\right)x,$$
(see [3]) where $\mu$ is the M{\"o}bius function, so heuristically it should only decrease the size of our set by at most a constant factor. Furthermore, it decreases the difficulty of the calculations tremendously, an idea that will become more apparent shortly. Plugging this in, we obtain 
$$\sum_{2\leq j \leq g_\epsilon(x^2)} \sum_{n\leq x} T_j(n^2)P_j(n)\geq\sum_{2\leq j \leq g_\epsilon(x^2)} \sum_{n\leq x} T_j(n^2)P_j(n)\mu^2(n)$$
Making the change of variables $j=3^r$ (where we know $r$ is a positive integer precisely because $n$ is square-free), we get 
$$\sum_{2\leq j \leq g_\epsilon(x^2)} \sum_{n\leq x} T_j(n)P_j(n)\mu^2(n)=\sum_{1\leq r \leq \log_3(g_\epsilon(x^2))} \sum_{n\leq x} T_{3^r}(n^2)P_{3^r}(n)\mu^2(n).$$
Notice from the identity 
$$\tau_2(n)=\prod_{p^\alpha||n} (\alpha+1)$$
that the statement   $\tau_2(n^2)=3^r$ is equivalent to the statement $\tau_2(n)=2^r$ if $n$ is square free, and thus we have 
$$\sum_{1\leq r \leq \log_3(g_\epsilon(x^2))} \sum_{n\leq x} T_{3^r}(n^2)P_{3^r}(n)\mu^2(n)=\sum_{1\leq r \leq \log_3(g_\epsilon(x^2))} \sum_{n\leq x} T_{2^r}(n)P_{3^r}(n)\mu^2(n)$$
Upon interpretation of each indicator function, one can see that the summand of the innermost sum is now just the indicator function of those $n$ which are less than or equal to $x$ and are a product of $r$ distinct primes all of which are greater than $3^r$. Utilizing this in the sum, we obtain 
\begin{equation}
\sum_{1\leq r \leq \log_3(g_\epsilon(x^2))} \sum_{n\leq x} T_{2^r}(n)P_{3^r}(n)\mu^2(n)=\sum_{1\leq r \leq \log_3(g_\epsilon(x^2))} \sum_{\substack{p_1 \times...\times p_r\leq x \\ 3^r<p_1<...<p_r}} 1
\end{equation}
Notice that if we did not restrict our sum to square-free numbers, we would have to include a third sum into the right hand side of (16) to sum over all $l$-tuples $(a_1,...,a_l)$ that satisfy $$\prod_i (a_i +1)=k$$ for each value of $k$, but since $n$ is square-free, the only tuples we need to concern ourselves with are those of the form $(2,...,2)$, significantly reducing the work involved.

Now, the innermost sum on the right hand side in (16) seems to be relatively easy (yet tedious) to calculate by means of splitting up the sum into $r$ 1-dimensional sums and dealing with each sum individually (as long as $r$ grows slowly enough with $x$ which can be controlled by truncating the outer sum earlier at the cost of a lower bound) using the Prime Number Theorem or Mertens' other theorem 
$$\sum_{p\leq x} \frac{1}{p}=\log(\log(x))+O(1)$$
(see [1, page 90] for a proof). This task, while seemingly doable, will not be further discussed in this paper and is left as an open problem to the reader. 

There are limitations to our entire method in general, of course. It is generally believed that the insertion of Lemma 1, the heart of the method, limits one from getting a bound of even the form $$|\{n\leq x: M(n)\neq0 \}|>\epsilon x$$ for any $\epsilon>0$. Thus, a new idea is needed for such a bound.

\section{Acknowledgments} I would like to thank Stephen Cochran for his unrelenting support and motivation throughout this rigorous process. Also, this paper would not have been written if not for the subtle insight and constant motivation from my good friend Todd Fellman. Last but most certainly not least, I would like to thank Dr. Xiao-dong Zhang, Dr. Stephen Locke, Dr. Tomas Schonbek, and Dr. Markus Schmidmeier for their priceless guidance in the editory process. I am forever indebted to you all. 

\section{References}
\begin{enumerate}
\item{Halberstam, H., \& Richert, H. E. (2013). \textit{Sieve methods}. Courier Corporation.}
\item{Apostol, T. M. (2013). \textit{Introduction to analytic number theory}. Springer Science \& Business Media.}
\item{Pawlewicz, J. (2011). Counting square-free numbers. \textit{arXiv preprint arXiv:1107.4890.}}
\end{enumerate}

\end{document}